\documentclass[11pt, reqno]{amsart}
\usepackage{amssymb}
\usepackage{verbatim}
\usepackage[vcentermath]{youngtab}
\usepackage{float}
\usepackage{stmaryrd}

\usepackage{times}
\usepackage[T1]{fontenc}
\usepackage{mathrsfs}
\usepackage{epsfig}
\usepackage{color}
\usepackage{array} 
\newcolumntype{L}{>{$}l<{$}}
\newcolumntype{R}{>{$}r<{$}}

\newtheorem{thm}{Theorem}[section]
\newtheorem{lemma}[thm]{Lemma}

\newtheorem{cor}[thm]{Corollary}
\newtheorem{prop}[thm]{Proposition}

\newtheorem{defi}[thm]{Definition}

\theoremstyle{definition}
\newtheorem*{defn}{Definition}

\numberwithin{equation}{section}

\def\ZZ{\mathbb{Z}}
\def\QQ{\mathbb{Q}}

\def\AA{\mathbb{A}}

\def\FF{\mathbb{F}}

\def\RR{\mathbb{R}}

\def\mathstrut{\vphantom(}

\newcommand\PConf{\mathrm{PConf}}

\def\sgn{\mathrm{sgn}}

\newcommand{\poly}{\mathrm{Poly}}
\newcommand{\irr}{\mathrm{Irr}}

\newcommand{\stir}[2]{\genfrac{\lbrack}{\rbrack}{0pt}{}{#1}{#2}}

\newcommand{\sfr}{\mathrm{sf}}
\def\multi#1#2{\ensuremath{\left(\kern-.3em\left(\genfrac{}{}{0pt}{}{#1}{#2}\right)\kern-.3em\right)}}

\begin{document}

\title{Liminal  Reciprocity and Factorization Statistics}

\author{Trevor Hyde}
\thanks{The author was partially supported by DMS-1701576.}
\address{Dept. of Mathematics\\
University of Michigan \\
Ann Arbor, MI 48109-1043\\
}
\email{tghyde@umich.edu}

\date{July 10th, 2018}

\maketitle

\begin{abstract}
    Let $M_{d,n}(q)$ denote the number of monic irreducible polynomials in $\FF_q[x_1, x_2, \ldots, x_n]$ of degree $d$. We show that for a fixed degree $d$, the sequence $M_{d,n}(q)$ converges coefficientwise to an explicitly determined rational function $M_{d,\infty}(q)$. The limit $M_{d,\infty}(q)$ is related to the classic necklace polynomial $M_{d,1}(q)$ by an involutive functional equation we call \emph{liminal reciprocity}. The limiting first moments of factorization statistics for squarefree polynomials are expressed in terms of symmetric group characters as a consequence of liminal reciprocity, giving a liminal analog of the \emph{twisted Grothendieck-Lefschetz formula} of Church, Ellenberg, and Farb. 
\end{abstract}

\section{Introduction}
\label{sec introduction}
Let $\FF_q$ be a field with $q$ elements. How many irreducible polynomials of degree $d$ are there in $\FF_q[x_1, x_2, \ldots, x_n]$? Let $M_{d,n}(q)$ denote the number of irreducible polynomials in $\FF_q[x_1, x_2, \ldots, x_n]$ of total degree $d$ which are monic with respect to some fixed monomial ordering ($M_{d,n}(q)$ is independent of the choice of monomial ordering. See Section \ref{section factorization} for details.) When $n = 1$, $M_{d,1}(q)$ is given by the \emph{$d$th necklace polynomial}
\begin{equation}
\label{necklace 1}
    M_{d,1}(q) := \frac{1}{d}\sum_{e\mid d}\mu(d/e) q^{e},
\end{equation}
where $\mu$ is the M\"{o}bius function. There does not appear to be a simple analog of \eqref{necklace 1} for $M_{d,n}(q)$ when $n > 1$. In Lemma \ref{lem1} $M_{d,n}(q)$ is shown to be a recursively computable polynomial in $q$ for all $n\geq 1$. The table below gives the low degree terms of $M_{3,n}(q)$ for small $n$.

\begin{center}
\begin{tabular}{|c|l|}
    \hline
    $n$ & $M_{3,n}(q)$\\
    \hline
    1 & ${-\frac{1}{3}}q + \frac{1}{3} q^{3}$\\
    \hline
    2  & ${-\frac{1}{3}} q  -\frac{1}{3} q^{2} + \frac{1}{3} q^{3}  -q^{5}  -\frac{2}{3} q^{6} + \ldots$\\
    \hline
    3  & ${-\frac{1}{3}} q  -\frac{1}{3}q^{2} + q^{4} + q^{5} + \frac{1}{3} q^{6} -q^{7} + \ldots$\\
    \hline
    4  & $ {-\frac{1}{3}} q  -\frac{1}{3} q^{2} + \frac{2}{3} q^{4} + 2 q^{5} + \frac{7}{3} q^{6} + 2 q^{7} + \ldots$\\
    \hline
    5  & $ {-\frac{1}{3}} q  -\frac{1}{3} q^{2} + \frac{2}{3} q^{4} + \frac{5}{3} q^{5} + \frac{10}{3} q^{6} + 4 q^{7} + \ldots$\\
    \hline
    6  & $ {-\frac{1}{3}} q  -\frac{1}{3} q^{2} + \frac{2}{3} q^{4} + \frac{5}{3} q^{5} + 3 q^{6} + 5 q^{7} + \ldots$\\
    \hline
    7  & $ {-\frac{1}{3}} q  -\frac{1}{3} q^{2} + \frac{2}{3} q^{4} + \frac{5}{3} q^{5} + 3 q^{6} + \frac{14}{3} q^{7} + \ldots$\\
    \hline
\end{tabular}
\end{center}

The table suggests that the sequence of polynomials $M_{3,n}(q)$ converge coefficientwise as the number of variables $n$ increases. We show that for any degree $d$, the sequence of polynomials $M_{d,n}(q)$ converges coefficientwise to a rational function $M_{d,\infty}(q)$ as $n\rightarrow\infty$. The limit $M_{d,\infty}(q)$ has an expression closely related to \eqref{necklace 1}.
\begin{thm}
\label{thm intro 1}
Let $M_{d,n}(q)$ be the number of irreducible degree $d$ polynomials in $\FF_q[x_1, x_2, \ldots, x_n]$ which are monic with respect to some fixed monomial ordering. Then $M_{d,n}(q)$ is a polynomial in $q$ and for each $d\geq 1$ the sequence of polynomials $M_{d,n}(q)$ converges coefficientwise (that is, with respect to the $q$-adic topology) in the formal power series ring $\QQ\llbracket q\rrbracket$ to the rational function
\[
    M_{d,\infty}(q) := -\frac{1}{d}\sum_{e\mid d} \mu(d/e)\left(\frac{1}{1-\frac{1}{q}}\right)^{e}.
\]
In particular $M_{d,\infty}(q)$ satisfies the functional equation,
\begin{equation}
\label{eqn recip intro}
    M_{d,\infty}(q) = - M_{d,1}\Big(\tfrac{1}{1-\frac{1}{q}}\Big).
\end{equation}
Furthermore the rate of convergence of $M_{d,n}(q)$ is bounded by the congruence
\[
    M_{d,n}(q) \equiv M_{d,\infty}(q) \bmod q^{n + 1}.
\]
\end{thm}

The fractional linear transformation $q \mapsto \frac{1}{1 - \frac{1}{q}}$ is an involution, hence \eqref{eqn recip intro} is equivalent to
\[
    M_{d,1}(q) = - M_{d,\infty}\Big(\tfrac{1}{1-\frac{1}{q}}\Big).
\]    
This involutive functional equation relating irreducible polynomial counts in one and infinitely many variables is the first instance of a phenomenon we call \emph{liminal reciprocity}.

\subsection{Liminal reciprocity for type polynomials}
Let $\poly_{d,n}(\FF_q)$ denote the set of polynomials in $\FF_q[x_1, x_2,\ldots, x_n]$ of total degree $d$ which are monic with respect to some fixed monomial ordering. Since the polynomial ring $\FF_q[x_1,x_2,\ldots, x_n]$ has unique factorization, each $f \in \poly_{d,n}(\FF_q)$ has a well-defined \emph{factorization type}. The factorization type of a polynomial $f\in \poly_{d,n}(\FF_q)$ is the partition $\lambda \vdash d$ given by the degrees of the $\FF_q$-irreducible factors of $f$. 

The factorization type does not record the multiplicities of individual factors, only the degrees of the irreducible factors. For example, the polynomials $x^2$ and $x(x + 1)$ both have factorization type $(1^2)$ since they each have two linear factors.

\begin{defn}
If $\lambda \vdash d$ is a partition, then the \emph{$\lambda$-type polynomial} $T_{\lambda,n}(q)$ is the number of elements in $\poly_{d,n}(\FF_q)$ with factorization type $\lambda$. Similarly the \emph{squarefree $\lambda$-type polynomial} $T_{\lambda,n}^\sfr(q)$ is the number of squarefree elements in $\poly_{d,n}(\FF_q)$ with factorization type $\lambda$. The type polynomials may be expressed in terms of $M_{d,n}(q)$ as
\[
    T_{\lambda,n}(q) = \prod_{j\geq 1} \multi{M_{j,n}(q)}{m_j(\lambda)}\hspace{.5in}
    T_{\lambda,n}^\sfr(q) = \prod_{j\geq 1} \binom{M_{j,n}(q)}{m_j(\lambda)},
\]
where $m_j(\lambda)$ is the number of parts of $\lambda$ of size $j$, $\binom{x}{m} = \frac{1}{m!}x(x-1)\cdots(x-m +1)$ is the usual binomial coefficient, and $\multi{x}{m} = \frac{1}{m!}x(x + 1)\cdots (x + m - 1)$. Recall that $\binom{x}{m}$ counts the number of subsets of size $m$ in a set of size $x$ and $\multi{x}{m}$ counts the number of subsets of size $m$ with repetition in a set of size $x$.
\end{defn}

It follows from Theorem \ref{thm intro 1} that the coefficientwise limits 
\[
    T_{\lambda,\infty}(q) = \lim_{n\rightarrow\infty}T_{\lambda,n}(q) \hspace{.5in}
    T_{\lambda,\infty}^\sfr(q) = \lim_{n\rightarrow\infty}T_{\lambda,n}^\sfr(q)
\]
converge to rational functions. Our next result is a version of liminial reciprocity for type polynomials.

\begin{thm}[Liminal reciprocity]
\label{thm type reciprocity}
Let $\lambda$ be a partition and let $\ell(\lambda) = \sum_{j\geq 1}m_j(\lambda)$ be the number of parts of $\lambda$. Then the following identities hold in $\QQ(q)$,
\begin{align*}
    T_{\lambda,\infty}(q) &= (-1)^{\ell(\lambda)} T_{\lambda,1}^\sfr\Big(\tfrac{1}{1-\frac{1}{q}}\Big)\\
    T_{\lambda,\infty}^\sfr(q) &= (-1)^{\ell(\lambda)} T_{\lambda,1}\Big(\tfrac{1}{1-\frac{1}{q}}\Big)
\end{align*}
\end{thm}

These identities are involutive in the sense that we can swap the $\infty$ and $1$ subscripts to get equivalent statements. The new feature appearing in Theorem \ref{thm type reciprocity} is the relationship between squarefree polynomials and general polynomials of a given factorization type. This connection is closely related to Stanley's \emph{combinatorial reciprocity phenomenon} \cite{Stanley 1} (see Section \ref{sec related} below.)

\subsection{Liminal first moments of squarefree factorization statistics}

A function $P$ defined on $\poly_{d,n}(\FF_q)$ is called a \emph{factorization statistic} if $P(f)$ depends only on the factorization type of $f$. In \cite{Hyde} we found a surprising connection between the first moments of factorization statistics on the set of univariate polynomials ($n = 1$) and the cohomology of point configurations in Euclidean space viewed as a representation of the symmetric group. See Section \ref{sec liminal reps} for precise definitions. Note that $\poly_{d,n}^\sfr(\FF_q)$ denotes the subset of squarefree polynomials in $\poly_{d,n}(\FF_q)$.

\begin{thm}[{\cite[Thm. 2.4, Thm. 2.5]{Hyde}}]
\label{thm intro univariate twisted gl}
Let $P$ be a factorization statistic, and let $\psi_d^k$, $\phi_d^k$ be the characters of the $S_d$-representations $H^{2k}(\PConf_d(\RR^3),\QQ)$ and $H^k(\PConf_d(\RR^2),\QQ)$ respectively. Then
\begin{align*}
    (1) \sum_{f\in \poly_{d,1}(\FF_q)} P(f) &= \sum_{k=0}^{d-1} \langle P, \psi_d^k \rangle q^{d-k}\\
    (2) \sum_{f\in \poly_{d,1}^\sfr(\FF_q)} P(f) &= \sum_{k=0}^{d-1} (-1)^k\langle P, \phi_d^k \rangle q^{d-k},
\end{align*}
where $\langle P, Q \rangle = \frac{1}{d!}\sum_{\tau\in S_d}P(\tau)Q(\tau)$ is the standard inner product of class functions on $S_d$.
\end{thm}

The squarefree case (2) of Theorem \ref{thm intro univariate twisted gl} is due to Church, Ellenberg, and Farb \cite[Prop. 4.1]{CEF}. The general polynomial case (1) was shown by the author \cite{Hyde} using different methods which also led to a new proof of the squarefree case. Theorem \ref{thm intro univariate twisted gl} provides a bridge between the arithmetic and combinatorics of factorization statistics on one hand and the geometry and representation theory of configuration spaces on the other.

Computations suggest there are not direct analogs of Theorem \ref{thm intro univariate twisted gl} for $n > 1$. However, an analog does emerge in the liminal squarefree case.

\begin{thm}
\label{thm intro liminal twisted GL}
Let $P$ be a factorization statistic, and let $\sigma_d^k$ be the character of the $S_d$-representation
\begin{equation}
\label{eqn intro rep def}
    \Sigma_d^k = \bigoplus_{j=k}^{d-1} \sgn_d \otimes H^{2j}(\PConf_d(\RR^3),\QQ)^{\oplus \binom{j}{k}}.
\end{equation}
Then for each $n$ the first moment $\sum_{f\in \poly_{d,n}^\sfr(\FF_q)} P(f)$ is a polynomial in $q$ and
\[
    \lim_{n\rightarrow\infty} \sum_{f\in \poly_{d,n}^\sfr(\FF_q)} P(f) = \frac{1}{(1 - q)^d}\sum_{k=0}^{d-1} (-1)^k \langle P, \sigma_d^k \rangle q^{d-k},
\]
where the limit is taken coefficientwise in $\QQ\llbracket q \rrbracket$.
\end{thm}

Since the limit in Theorem \ref{thm intro liminal twisted GL} is taken coefficientwise, the representation theoretic interpretation of first moments manifests for sufficiently large $n$. For example, let $L$ be the \emph{linear factor} statistic where $L(f)$ is the number of linear factors of $f$; the following table shows the first moment of $L$ on $\poly_{3,n}^\sfr(\FF_q)$ scaled by $(1 - q)^3$.

\begin{center}
\begin{tabular}{|c|l|}
\hline
$n$ & $(1-q)^3 \sum_{f\in \poly_{3,n}^\sfr(\FF_q)}L(f)$\\
\hline
$1$ & $q - 5q^2 + 10q^3 - 10q^4 + 5q^5 - q^6$ \\
\hline
$2$ & $q - 4q^2 + 2q^3 + 7q^4 - 6q^5 - 3q^6 + 2q^7 + q^8 + q^9 - q^{10}$ \\
\hline
$3$ & $q - 4q^2 + 3q^3 - q^4 + 7q^5 -6q^6 - 3q^8 + 3q^9 - q^{11} + q^{12} + q^{14} - q^{15}$ \\
\hline
$4$ & $q - 4q^2 + 3q^3 - q^5 + 7q^6 -6q^7 - 3q^{10} + 3q^{11} - q^{16} + q^{17} + q^{20} - q^{21}$ \\
\hline
$5$ & $q - 4q^2 + 3q^3 - q^6 + 7q^7 -6q^8 - 3q^{12} + 3q^{13} - q^{22} + q^{23} + q^{27} - q^{28}$ \\
\hline
\end{tabular}
\end{center}
From this table and the convergence bound in Theorem \ref{thm intro 1} we conclude that
\[
    \sum_{f\in \poly_{3,n}^\sfr(\FF_q)} L(f) = \frac{q - 4q^2 + 3q^3 + O(q^{n+1})}{(1 - q)^3}.
\]
It then follows from Theorem \ref{thm intro liminal twisted GL} that
\[
    \langle L, \sigma_3^2 \rangle = 1 \hspace{1em} \langle L, \sigma_3^1 \rangle = 4 \hspace{1em} \langle L, \sigma_3^0 \rangle = 3.
\]
Note that these inner products are positive integers: this reflects that $L$, viewed as a class function of the symmetric group, is the character of the standard permutation representation.

The table above also illustrates a higher stability in the coefficients. For example, the coefficient of $q^{n+2}$ is 7 in the numerator of the first moment of $L$ for all $n\geq 2$. Since these exponents grow with $n$, these terms vanish in the limit as $n\rightarrow \infty$. This phenomenon persists more generally and we hope to address it in a future project.

Liminal reciprocity gives a new method to compute the expected values of factorization statistics for univariate polynomials. As an example application we compute the expected value of the sign function $\sgn_d$, where $\sgn_d(\lambda) = (-1)^{d - \sum_{j\geq 1}m_j(\lambda)}$.

\begin{prop}
\label{sgn prop intro}
Let $d\geq 1$.
\begin{enumerate}
    \item The expected value $E_{d,1}(\sgn_d)$ of the sign statistic on the set $\poly_{d,1}(\FF_q)$ is given by
    \[
        E_{d,1}(\sgn_d) := \frac{1}{P_{d,1}(q)}\sum_{f\in \poly_{d,1}(\FF_q)} \sgn_d(f) =  \frac{1}{q^{\lfloor d/2 \rfloor}}.
    \]
    \item The limiting expected value $E_{d,\infty}^\sfr(\sgn_d)$ of the sign statistic on the set $\poly_{d,n}^\sfr(\FF_q)$ as $n\rightarrow\infty$ is given by
    \[
        E_{d,\infty}^\sfr(\sgn_d) := \lim_{n\rightarrow\infty} \frac{1}{P_{d,n}^\sfr(q)}\sum_{f\in \poly_{d,n}^\sfr(\FF_q)} \sgn_d(f) = \left(\frac{1}{1-\frac{1}{q}}\right)^{\lfloor d/2 \rfloor},
    \]
    where the limit is taken $1/q$-adically.
\end{enumerate}
\end{prop}

This result is equivalent to a result of Carlitz arrived at by other means. See Proposition \ref{sgn prop} and the discussion that follows.

\subsection{Related work}
\label{sec related}

Carlitz \cite{Carlitz 1, Carlitz 2} studied the asymptotic behavior of $M_{d,n}(q)$ for $n\geq 1$. In the language of this paper his main result is as follows.
\begin{thm}[{\cite[Sec. 3.]{Carlitz 1}}]
\label{thm Carlitz}
For $d, n \geq 1$, let $m_{d,n} := \deg M_{d,n}(q)$. Then $m_{d,n} = \binom{d + n}{d} - 1$ and the sequence $M_{d,n}(q)/q^{m_{d,n}}$ of polynomials in $1/q$ converges coefficientwise in $\QQ\llbracket \frac{1}{q}\rrbracket$ to
\[
    \lim_{n\rightarrow\infty}\frac{M_{d,n}(q)}{q^{m_{d,n}}} = \frac{1}{1 - \frac{1}{q}}.
\]
\end{thm}
This work was subsequently refined and extended in \cite{Bodin, Cohen, Hou Mullen, Ziegler, Wan}. Our Theorem \ref{thm intro 1} may be interpreted as a determination of the $q$-adic asymptotics of $M_{d,n}(q)$ as $n\rightarrow\infty$. In other words Carlitz studied the limiting behavior of the leading terms of $M_{d,n}(q)$ and we study the limiting behavior of the low degree terms.

The liminal reciprocity identities (Theorem \ref{thm intro 1} and Theorem \ref{thm type reciprocity}) were discovered empirically. We do not know the proper context for these results. The proof of the liminal reciprocity for type polynomials (Theorem \ref{thm type reciprocity}) passes through a well-known example of Stanley's \emph{combinatorial reciprocity phenomenon} \cite[Ex. 1.1]{Stanley 1}. Combinatorial reciprocity is a family of dualities between related combinatorial problems which concretely manifests as functional equations similar in form to our liminal reciprocity identities. However, the precise relationship between liminal and combinatorial reciprocity remains unclear. Are there other examples of liminal reciprocity?

The relationship between the liminal first moments of squarefree factorization statistics and representations of the symmetric group parallels our results in \cite{Hyde}. Church, Ellenberg, and Farb \cite{CEF} established the connection between first moments of squarefree factorization statistics for univariate polynomials and the cohomology of point configurations in $\RR^2$ with their \emph{twisted Grothendieck-Lefschetz formula} for squarefree polynomials. They deduce the asymptotic stability of first moments (as $d \rightarrow \infty$) as a consequence of \emph{representation stability}. We extend this connection to general univariate polynomials in \cite[Thm. 2.7]{Hyde}. However, this connection does not extend to liminal first moments; the representations $\Sigma_d^k$ does not exhibit representation stability.

The results in \cite{Hyde} are expressed in terms of expected values of factorization statistics. In this paper we focus on first moments as they lead to a cleaner statement for Theorem \ref{thm intro liminal twisted GL}. The only difference between expected values and first moments of factorization statistics is whether or not one divides by the ``total mass'' of the space of polynomials considered. This difference is simply a factor of $q^d$ for general univariate polynomials, but is more subtle for squarefree polynomials and multivariate polynomials as it affects the family of characters determining the coefficients. The equivalence between Theorem \ref{thm intro univariate twisted gl} (2) and \cite[Thm. 2.5]{Hyde} follows from \cite[Prop. 4.2]{Hyde Lagarias}. Alternatively, Theorem \ref{thm intro univariate twisted gl} (2) appears as stated in \cite[Prop. 4.1]{CEF}.

Finally we note that by virtue of treating arbitrary factorization statistics $P$ our results also provide information on higher moments of $P$ (since the $k$th moment of $P$ is the first moment of $P^k$.) 

\subsection{Acknowledgements}
The author thanks Weiyan Chen, Nir Gadish, Ofir Gorodetsky, Jeff Lagarias, Bob Lutz, John Stembridge, Phil Tosteson, Michael Zieve, and the anonymous referee for helpful conversations and suggestions on the manuscript.


\section{Polynomial factorization statistics}
\label{section factorization}

Let $\FF_q$ be a finite field. Fix some monomial ordering on $\FF_q[x_1, x_2, \ldots, x_n]$. A polynomial $f \in \FF_q[x_1, x_2, \ldots, x_n]$ is \emph{monic} with respect to this monomial ordering if the leading coefficient of $f$ is 1. Let $\poly_{d,n}(\FF_q)$ be the set of all degree $d$ polynomials in $\FF_q[x_1, x_2, \ldots, x_n]$ which are monic with respect to the monomial ordering. Note that the size of $\poly_{d,n}(\FF_q)$ is independent of the choice of monomial ordering. For each $m\geq 1$ let $\poly_{d,n}^m(\FF_q)\subseteq \poly_{d,n}(\FF_q)$ be the subset of those polynomials with all factors of multiplicity at most $m$. There is a filtration
\[
    \poly_{d,n}^\sfr(\FF_q) := \poly_{d,n}^1(\FF_q) \subseteq \poly_{d,n}^2(\FF_q) \subseteq \poly_{d,n}^3(\FF_q) \subseteq \ldots \subseteq \poly_{d,n}(\FF_q),
\]
where $\poly_{d,n}^\sfr(\FF_q)$ is the set of the squarefree polynomials.

Recall that $\FF_q[x_1, x_2, \ldots, x_n]$ is a unique factorization domain, hence every element of $\poly_{d,n}(\FF_q)$ has a unique factorization as a product of irreducible monic polynomials. The \emph{factorization type} of $f \in \poly_{d,n}(\FF_q)$ is the partition of $d$ given by the degrees of the irreducible factors of $f$. If $\lambda$ is a partition of $d$, then $\poly_{\lambda,n}(\FF_q)$ denotes the set of all $f \in \poly_{d,n}(\FF_q)$ with factorization type $\lambda$. For $m \geq 1$, let $\poly_{\lambda,n}^m(\FF_q) := \poly_{d,n}^m(\FF_q) \cap \poly_{\lambda,n}(\FF_q)$. If $\lambda = (d)$ is the partition with one part, let $\irr_{d,n}(\FF_q) := \poly_{(d),n}(\FF_q)$ be the set of monic, irreducible, total degree $d$ polynomials.

Lemma \ref{lem1} shows that the cardinality of each of the sets just defined is given by a polynomial in the size of the field $q$.

\begin{lemma}
\label{lem1}
For any $d,n\geq 1$,
\begin{enumerate}
    \item $|\poly_{d,n}(\FF_q)| = P_{d,n}(q)$, where
    \[
        P_{d,n}(q) = \frac{q^{\binom{d + n}{n}} - q^{\binom{d + n - 1}{n}}}{q-1} = q^{\binom{d + n - 1}{n}} \frac{q^{\binom{d + n - 1}{n- 1}} - 1}{q - 1}.
    \]
    \item $M_{d,n}(q)$ is a polynomial of $q$ with rational coefficients.

    \item For every partition $\lambda \vdash d$,
    \begin{align*}
        |\poly_{\lambda,n}(\FF_q)| &= T_{\lambda,n}(q) := \prod_{j\geq 1}\multi{M_{j,n}(q)}{m_j(\lambda)},\\
        |\poly_{\lambda,n}^\sfr(\FF_q)| &= T_{\lambda,n}^\sfr(q) := \prod_{j\geq 1} \binom{M_{j,n}(q)}{m_j(\lambda)}.
    \end{align*}
    where $\multi{x}{m} := \binom{x + m - 1}{m}$ is the number of subsets with repetition of size $m$ chosen from an $x$ element set.
\end{enumerate}
\end{lemma}

\begin{proof}
(1) There are $q^{\binom{d + n}{n}}$ polynomials in $n$ variables of degree at most $d$. Hence there are $q^{\binom{d+n}{n}} - q^{\binom{d+n-1}{n}}$ polynomials in $n$ variables of degree exactly $d$. After choosing a monomial order, every degree $d$ polynomial has a nonzero leading coefficient. Therefore the total number of degree $d$ monic polynomials in $n$ variables is
\[
    |\poly_{d,n}(\FF_q)| = \frac{q^{\binom{d+n}{n}} - q^{\binom{d+n-1}{n}}}{q-1}.
\]

\noindent (2) We proceed by induction on $d$ to show that
\[
    |\irr_{d,n}(\FF_q)| = M_{d,n}(q)
\]
for some polynomial $M_{d,n}(x) \in \QQ[x]$. If $d = 1$, then all polynomials are irreducible, hence
\[  
    |\irr_{1,n}(\FF_q)| = |\poly_{1,n}(\FF_q)| = q\frac{q^n - 1}{q-1},
\]
So $M_{1,n}(q) = q\frac{q^n - 1}{q-1}$. Suppose our claim were true for all degrees less than $d>1$. By unique factorization, the total number of polynomials with factorization type $\lambda$ is
\begin{equation}
\label{eq1}
    |\poly_{\lambda,n}(\FF_q)| = \prod_{j\geq 1}\multi{|\irr_{j,n}(\FF_q)|}{m_j(\lambda)}.
\end{equation}
Counting elements on both sides of the decomposition
\[
    \poly_{d,n}(\FF_q) = \bigsqcup_{\lambda \vdash d} \poly_{\lambda,n}(\FF_q),
\]
gives
\[
    P_{d,n}(q) = |\irr_{d,n}(\FF_q)| + \sum_{\substack{\lambda \vdash d\\ \lambda \neq (d)}} |\poly_{\lambda,n}(\FF_q)|.
\]
If $\lambda \neq (d)$, then all parts $j$ of $\lambda$ are smaller than $d$, which by our inductive hypothesis implies that $|\irr_{j,n}(\FF_q)| = M_{j,n}(q)$ for all such $j$. Thus
\[
   |\irr_{d,n}(\FF_q)| = M_{d,n}(q) := P_{d,n}(q) - \sum_{\substack{\lambda \vdash d\\\lambda \neq (d)}}\prod_{j\geq 1} \multi{M_{j,n}(q)}{m_j(\lambda)}.
\]
Finally, (3) follows from \eqref{eq1} and (2).
\end{proof}
The definitions of the polynomials appearing in Lemma \ref{lem1} are collected here for the reader's convenience.

\begin{defi}
\label{def fact}
Let $d,n\geq 1$ and $\lambda \vdash d$, then
\begin{align*}
    P_{d,n}(q) &=  \frac{q^{\binom{d + n}{n}} - q^{\binom{d + n - 1}{n}}}{q-1} = q^{\binom{d + n - 1}{n}} \frac{q^{\binom{d + n - 1}{n- 1}} - 1}{q - 1}\\
    M_{d,n}(q) &= |\irr_{d,n}(\FF_q)| = |\poly_{(d),n}(\FF_q)|\\
    T_{\lambda,n}(q) &= |\poly_{\lambda,n}(\FF_q)| = \prod_{j\geq 1} \multi{M_{j,n}(q)}{m_j(\lambda)}\\
    T_{\lambda,n}^m(q) &= |\poly_{\lambda,n}^m(\FF_q)|\\
    T_{\lambda,n}^\sfr(q) &= T_{\lambda,n}^1(q) = |\poly_{\lambda,n}^\sfr(\FF_q)| = \prod_{j\geq 1} \binom{M_{j,n}(q)}{m_j(\lambda)}\\
    P_{d,n}^m(q) &= |\poly_{d,n}^m(\FF_q)| = \sum_{\lambda \vdash d} T_{\lambda,n}^m(q),
\end{align*}
where $d$ represents \textbf{d}egree, $n$ the \textbf{n}umber of variables, and $m$ the maximum \textbf{m}ultiplicity of a factor.
\end{defi}

\noindent There is a well-known formula \cite[Cor. 2.1]{Rosen} for $M_{d,1}(q)$ given by counting elements in $\FF_{q^d}$ by the field they generate,
\begin{equation}
\label{eq2}
    M_{d,1}(q) = \frac{1}{d}\sum_{e\mid d} \mu(e) q^{d/e}.
\end{equation}
The value of $M_{d,1}(k)$ for an integer $k \geq 1$ has a combinatorial interpretation as the number of aperiodic necklaces made with beads of $k$ colors. For this reason, $M_{d,1}(q)$ is known as the $d$th \emph{necklace polynomial}. There is no apparent analog of \eqref{eq2} nor a combinatorial interpretation for $M_{d,n}(k)$ when $n > 1$. Instead $M_{d,n}(q)$ may be computed inductively as in the proof of Lemma \ref{lem1}:
\begin{align*}
    M_{1,n}(q) &= P_{1,n}(q) = q\frac{q^n - 1}{q-1}\\
    M_{d,n}(q) &= P_{d,n}(q) - \sum_{\substack{\lambda \vdash d \\ \lambda \neq [d]}} T_{\lambda,n}(q).
\end{align*}

Our next result shows that all the polynomials listed in Definition \ref{def fact} converge coefficientwise to rational functions in the ring of formal power series $\QQ\llbracket q \rrbracket$ as the number of variables $n$ tends to infinity. Recall that coefficientwise convergence in $\QQ\llbracket q \rrbracket$ is equivalent to convergence with respect to the $q$-adic topology. All coefficientwise limits are taken with respect to the $q$-adic topology.

\begin{thm}
\label{thm converge}
Let $d\geq 1$. Then,
\begin{enumerate}
    \item The sequence $P_{d,n}(q)$ converges coefficientwise in $\QQ\llbracket q \rrbracket$ to
    \[
        P_{d,\infty}(q) = \lim_{n\rightarrow \infty} P_{d,n}(q) = \begin{cases} -\frac{1}{1-\frac{1}{q}} & d = 1\\ 0 & d > 1. \end{cases}
    \]
    \item For $m\geq 1$ the sequence $P_{d,n}^m(q)$ converges coefficientwise in $\QQ\llbracket q \rrbracket$ to
    \[ 
        P_{d,\infty}^m(q) = \lim_{n\rightarrow\infty}P_{d,n}^m(q) = \begin{cases}-\Big(\frac{1}{1-\frac{1}{q}}\Big)^k & d = (m+1)k - m\\ \Big(\frac{1}{1-\frac{1}{q}}\Big)^k & d = (m+1)k\\ 0 & d \not\equiv 0,1 \bmod m+1.\end{cases}
    \]

    \noindent In particular, if $m = 1$, then
    \[
        P_{d,\infty}^\sfr(q) = (-1)^d \Big(\tfrac{1}{1-\frac{1}{q}}\Big)^{\lfloor \frac{d+1}{2}\rfloor}.
    \]
    \item For all partitions $\lambda \vdash d$ and $m\geq 1$ the sequences $M_{d,n}(q)$, $T_{\lambda,n}(q)$, and $T_{\lambda,n}^m(q)$ converge coefficientwise in $\QQ\llbracket q \rrbracket$ to rational functions as $n\rightarrow\infty$. Furthermore,
        \begin{align*}
            T_{\lambda,\infty}(q) &= \prod_{j\geq 1} \multi{M_{j,\infty}(q)}{m_j(\lambda)}\\
            T_{\lambda,\infty}^\sfr(q) &= \prod_{j\geq 1}\binom{M_{j,\infty}(q)}{m_j(\lambda)}.
        \end{align*}
\end{enumerate}
\end{thm}

\begin{proof}
(1) By Lemma \ref{lem1}
\[
    P_{d,n}(q) = q^{\binom{d+n-1}{n}}\frac{q^{\binom{d + n - 1}{n-1}} - 1}{q-1}.
\]
For $d = 1$ this simplifies to
\[
    P_{1,n}(q) = q\frac{q^n -1}{q - 1}.
\]
Since $\lim_{n\rightarrow\infty}q^n = 0$ in the $\QQ\llbracket q \rrbracket$, it follows that
\[
    P_{1,\infty}(q) = \lim_{n\rightarrow\infty} q \frac{q^n - 1}{q-1} =  -\frac{q}{q-1} = -\frac{1}{1 - \frac{1}{q}}.
\]
If $d>1$, then $\lim_{n\rightarrow\infty}\binom{d + n - 1}{n} = \infty$. Thus
\[
    P_{d,\infty}(q) = \lim_{n\rightarrow\infty} q^{\binom{d+n-1}{n}}\frac{q^{\binom{d + n - 1}{n-1}} - 1}{q-1} = 0.
\]
(2) Consider the generating functions
\begin{align*}
    Z(T_n^m,t) &= \sum_{d\geq 0}\sum_{\lambda\vdash d} T_{\lambda,n}^m(q) t^d = \sum_{d\geq 0}P_{d,n}^m(q) t^d\\
    Z(T_n, t) &= \sum_{d\geq 0}\sum_{\lambda\vdash d} T_{\lambda,n}(q) t^d = \sum_{d\geq 0} P_{d,n}(q) t^d.
\end{align*}

Recall that the binomial theorem allows us to exponentiate $1 + t$ or $\frac{1}{1 - t}$ by any element $m$ of a ring $R$ in $R\llbracket t \rrbracket$ by
\begin{align*}
    (1 + t)^m &:= \sum_{d\geq 0}\binom{m}{d}t^d\\
    \left(\frac{1}{1 - t}\right)^m &:= \sum_{d\geq 0}\multi{m}{d}t^d.
\end{align*}

The following product formulas follow by unique factorization in $\FF_q[x_1, x_2, \ldots, x_n]$,
\begin{align*}
    Z(T_n^m, t) &= \prod_{j\geq 1}(1 + t^j + t^{2j} + \ldots + t^{mj})^{M_{j,n}(q)} = \prod_{j\geq 1}\left(\frac{1 - t^{(m+1)j}}{1 - t^j}\right)^{M_{j,n}(q)}\\
    Z(T_n, t) &= \prod_{j\geq 1}\left(\frac{1}{1 - t^j}\right)^{M_{j,n}(q)}.
\end{align*}
Hence $Z(T_n,t) = Z(T_n,t^{m+1})Z(T_n^m,t)$. The coefficients of $t^d$ for $d\geq 0$ in this identity are polynomials in $\QQ\llbracket q \rrbracket$ which converge $q$-adically as $n\rightarrow\infty$. Taking a limit $t$-coefficientwise as $n\rightarrow\infty$, (1) implies that
\[
    1 - \tfrac{1}{1-\frac{1}{q}}t = Z(T_\infty,t) = Z(T_\infty,t^{m+1})Z(T_\infty^m,t) = \big(1 - \tfrac{1}{1-\frac{1}{q}}t^{m+1}\big)\sum_{d\geq 0} P_{d,\infty}^m(q) t^d.
\]
Comparing coefficients we conclude that
\[
    P_{d + m + 1,\infty}^m(q) = \frac{1}{1-\frac{1}{q}}P_{d,\infty}^m(q)
\]
for all $d\geq 0$, together with the initial values
\begin{align*}
    P_{0,\infty}^m(q) &= 1\\
    P_{1,\infty}^m(q) &= -\frac{1}{1-\frac{1}{q}}\\
    P_{d,\infty}^m(q) &= 0 \text{ for } 1 < d \leq m.
\end{align*}
Then (2) follows by induction.\\

\noindent (3) It suffices to prove that for every $d\geq 1$ the sequence $M_{d,n}(q)$ converges $q$-adically to a rational function, the other claims follow by the explicit formulas given in Definition \ref{def fact} and continuity. Recall the recursive formulas for $M_{d,n}(q)$ used in the proof of Lemma \ref{lem1}. For all $d,n\geq 1$,
\begin{align*}
    M_{1,n}(q) &= P_{1,n}(q)\\
    M_{d,n}(q) &= P_{d,n}(q) - \sum_{\substack{\lambda \vdash d \\ \lambda \neq [d]}} \prod_{j\geq 1} \multi{M_{j,n}(q)}{m_j(\lambda)}.
\end{align*}
Taking coefficientwise limits as $n\rightarrow \infty$ using (1) gives
\begin{align*}
    M_{1,\infty}(q) &= P_{1,\infty}(q) =  - \frac{1}{1-\frac{1}{q}},\\
    M_{d,\infty}(q) &=  - \sum_{\substack{\lambda \vdash d \\ \lambda \neq [d]}} \prod_{j\geq 1} \multi{M_{j,\infty}(q)}{m_j(\lambda)}.
\end{align*}
It follows by induction that $M_{d,\infty}(q)$ is a rational function of $q$ for all $d\geq 1$.
\end{proof}

There is a surprising relationship between the number of irreducible polynomials in one variable $M_{d,1}(q)$ and the limit $M_{d,\infty}(q)$ which gives us an explicit formula for $M_{d,\infty}(q)$. This relationship takes the form of an involutive functional equation we call \emph{liminal reciprocity}.

\begin{thm}[Liminal reciprocity]
\label{liminal necklace}
For all $d \geq 1$,
\[
    M_{d,\infty}(q) = - M_{d,1}\Big(\tfrac{1}{1-\frac{1}{q}}\Big).
\]
More explicitly,
\[
    M_{d,\infty}(q) = -\frac{1}{d}\sum_{e\mid d} \mu(d/e)\bigg(\frac{1}{1 - \frac{1}{q}}\bigg)^e.
\]
\end{thm}

\begin{proof}
Recall the generating function $Z(T_n,t)$ used in the proof of Theorem \ref{thm converge} (2),
\[
    Z(T_n, t) = \sum_{d\geq 0}P_{d,n}(q) t^d = \prod_{j\geq 1}\left(\frac{1}{1 - t^j}\right)^{M_{j,n}(q)}
\]
Theorem \ref{thm converge} (1) implies that the $t$-coefficientwise limit as $n\rightarrow\infty$ is simply
\begin{equation}
\label{eqn gen id}
    1 - \tfrac{1}{1-\frac{1}{q}}t = \prod_{d\geq 1}\left(\frac{1}{1-t^d}\right)^{M_{d,\infty}(q)}.
\end{equation}
Consider the well-known \emph{cyclotomic identity} \cite{Met Rota}, or equivalently the Euler product formula for the Hasse-Weil zeta function of $\AA^1(\FF_q)$,
\begin{equation}
\label{eqn cyclo}
    \frac{1}{1 - qt} = \prod_{d\geq 1}\left(\frac{1}{1 - t^d}\right)^{M_{d,1}(q)}.
\end{equation}
Substituting $q \mapsto \frac{1}{1 - \frac{1}{q}}$ and taking reciprocals in \eqref{eqn cyclo} gives
\[
    1 - \tfrac{1}{1-\frac{1}{q}}t = \prod_{d\geq 1}\left(\frac{1}{1-t^d}\right)^{-M_{d,1}\big(\frac{1}{1 - \frac{1}{q}}\big)}.
\]
Comparing exponents with \eqref{eqn gen id} we conclude that
\[
    M_{d,\infty}(q) = - M_{d,1}\Big(\tfrac{1}{1-\frac{1}{q}}\Big).
\]
\end{proof}

The rate of $q$-adic convergence of $M_{d,n}(q)$ may be determined from the proof of Theorem \ref{liminal necklace}.

\begin{cor}
For all $d\geq 1$,
\[
    M_{d,n}(q) \equiv M_{d,\infty}(q) \bmod q^{n+1}.
\]
\end{cor}

\begin{proof}
Recall that
\[
    P_{d,n}(q) = q^{\binom{d+n-1}{n}}\frac{q^{\binom{d+n-1}{n-1}} - 1}{q-1}.
\]
Since $\binom{d + n - 1}{n} \geq n + 1$ for $d\geq 2$ and
\[
    P_{1,n}(q) = \frac{q^{n+1}-q}{q-1} \equiv - \frac{1}{1 - \frac{1}{q}} \bmod q^{n+1},
\]
it follows that
\[
    \sum_{d\geq 0}P_{d,n}(q)t^d \equiv 1 - \tfrac{1}{1-\frac{1}{q}}t \bmod q^{n+1}.
\]
Thus
\begin{align*}
    \prod_{d\geq 1}\left(\frac{1}{1-t^d}\right)^{M_{d,n}(q)} &= \sum_{d\geq 0}P_{d,n}(q)t^d\\
    &\equiv 1 - \tfrac{1}{1-\frac{1}{q}}t \bmod q^{n+1}\\
    &\equiv \prod_{d\geq 1}\left(\frac{1}{1-t^d}\right)^{M_{d,\infty}(q)} \bmod q^{n+1}
\end{align*}
and thus
\[
    M_{d,n}(q) \equiv M_{d,\infty}(q) \bmod q^{n+1}.
\]
\end{proof}

Notice that the fractional linear transformation $q \longmapsto \frac{1}{1-\frac{1}{q}}$ is an involution. Thus Theorem \ref{liminal necklace} is equivalent to
\[
    M_{d,1}(q) = - M_{d,\infty}\Big(\tfrac{1}{1-\frac{1}{q}}\Big).
\]
Our next result combines the liminal reciprocity relating $M_{d,1}(q)$ and $M_{d,\infty}(q)$ with the \emph{combinatorial reciprocity} identity
\begin{equation}
\label{eqn combo recip}
    \binom{-x}{m} = (-1)^m\multi{x\mathstrut}{m},
\end{equation}
to deduce a striking relationship between factorization statistics of polynomials when $n = 1$ and $n = \infty$.

\begin{thm}[Liminal reciprocity]
\label{thm main reciprocity}
For any partition $\lambda$, let $\ell(\lambda) = \sum_{j\geq 1}m_j(\lambda)$ denote the number of parts of $\lambda$. Then
\begin{align*}
    T_{\lambda,\infty}^\sfr(q) &= (-1)^{\ell(\lambda)} T_{\lambda,1}\Big(\tfrac{1}{1-\frac{1}{q}}\Big),\\
    T_{\lambda,\infty}(q) &= (-1)^{\ell(\lambda)} T_{\lambda,1}^\sfr\Big(\tfrac{1}{1-\frac{1}{q}}\Big).
\end{align*}
\end{thm}

\begin{proof}
Theorem \ref{thm converge} (3), Theorem \ref{liminal necklace}, and the combinatorial reciprocity identity \eqref{eqn combo recip} imply that
\begin{align*}
    T_{\lambda,\infty}^\sfr(q) &= \prod_{j\geq 1}\binom{M_{j,\infty}(q)}{m_j(\lambda)}\\
    &= \prod_{j\geq 1}\binom{-M_{j,1}\Big(\tfrac{1}{1-\frac{1}{q}}\Big)}{m_j(\lambda)}\\
    &= \prod_{j\geq 1}(-1)^{m_j(\lambda)}\multi{M_{j,1}\Big(\tfrac{1}{1-\frac{1}{q}}\Big)}{m_j(\lambda)}\\
    &= (-1)^{\ell(\lambda)} T_{\lambda,1}\Big(\tfrac{1}{1-\frac{1}{q}}\Big).
\end{align*}
The second identity follows from a parallel computation noting that \eqref{eqn combo recip} is equivalent to
\[
    \multi{-x}{m} = (-1)^m\binom{x}{m}.
\]
\end{proof}

The liminal reciprocity identity
\[
    T_{\lambda,\infty}^\sfr(q) = (-1)^{\ell(\lambda)}T_{\lambda,1}\Big(\tfrac{1}{1-\frac{1}{q}}\Big)
\]
relates the limiting number of squarefree polynomials with factorization type $\lambda$ in $\FF_q[x_1,x_2,\ldots,x_n]$ as $n\rightarrow \infty$ to the number of polynomials $\FF_q[x]$ with factorization type $\lambda$ with no restrictions on factor multiplicity. This relationship is, to us, rather mysterious. It would be interesting to find a conceptual explanation for this relationship between infinite and one dimensional factorization statistics.

\section{Liminal first moments of squarefree factorization statistics}
\label{sec liminal reps}
A \emph{factorization statistic} is a function $P$ defined on $\poly_{d,n}(\FF_q)$ such that $P(f)$ only depends on the factorization type of $f\in \poly_{d,n}(\FF_q)$. Equivalently, $P$ is a function defined on the partitions of the degree $d$, or a class function of the symmetric group $S_d$. In \cite{Hyde} we determined explicit formulas for the first moments of factorization statistics on $\poly_{d,1}(\FF_q)$ and $\poly_{d,1}^\sfr(\FF_q)$ in terms of the characters of symmetric group representations related to the cohomology of point configurations in Euclidean space.

\begin{thm}[{\cite[Thm. 2.4, Thm. 2.5]{Hyde}, \cite[Prop. 4.1]{CEF}}]
\label{thm univariate twisted gl}
Let $P$ be a factorization statistic, and let $\psi_d^k$, $\phi_d^k$ be the characters of the $S_d$-representations $H^{2k}(\PConf_d(\RR^3),\QQ)$ and $H^k(\PConf_d(\RR^2),\QQ)$ respectively. Then
\begin{align*}
    (1) \sum_{f\in \poly_{d,1}(\FF_q)} P(f) &= \sum_{k=0}^{d-1} \langle P, \psi_d^k \rangle q^{d-k}\\
    (2) \sum_{f\in \poly_{d,1}^\sfr(\FF_q)} P(f) &= \sum_{k=0}^{d-1} (-1)^k\langle P, \phi_d^k \rangle q^{d-k},
\end{align*}
where $\langle P, \psi_d^k\rangle = \frac{1}{d!}\sum_{\tau \in S_d}P(\tau)\psi_d^k(\tau)$ is the standard inner product of class functions on $S_d$.
\end{thm}

The identity (2) was first shown by Church, Ellenberg, and Farb \cite[Prop. 4.1]{CEF} using algebro-geometric methods including the Grothendieck-Lefschetz trace formula. They called this identity the \emph{twisted Grothendieck-Lefschetz formula}. We gave a new proof in \cite[Thm. 2.5]{Hyde} using a generating function argument. Our results in \cite{Hyde} are stated in terms of expected values instead of first moments; this distinction has little effect in the $\poly_{d,1}(\FF_q)$ case, but does change the family of representations in the squarefree case $\poly_{d,1}^\sfr(\FF_q)$: since the total number of degree $d$ squarefree polynomials in $\FF_q[x]$ is $P_{d,1}^\sfr(q) = q^d - q^{d-1}$, dividing the first moment of a factorization statistic by $q^d - q^{d-1}$ results in a polynomial with different coefficients. This version of (2) appears in \cite[Prop. 4.1]{CEF}.

The next result combines Theorem \ref{thm univariate twisted gl} with liminal reciprocity to express the limiting first moments of squarefree factorization statistics in terms of characters of symmetric group representations.

\begin{thm}
\label{thm liminal twisted GL}
Let $P$ be a factorization statistic, and let $\sigma_d^k$ be the character of the $S_d$-representation
\begin{equation}
\label{eqn rep def}
    \Sigma_d^k = \bigoplus_{j=k}^{d-1} \sgn_d \otimes H^{2j}(\PConf_d(\RR^3),\QQ)^{\oplus \binom{j}{k}}.
\end{equation}
Then
\[
    \lim_{n\rightarrow\infty} \sum_{f\in \poly_{d,n}^\sfr(\FF_q)} P(f) = \frac{1}{(1 - q)^d}\sum_{k=0}^d (-1)^k \langle P, \sigma_d^k \rangle q^{d-k}.
\]
\end{thm}

Theorem \ref{thm liminal twisted GL} follows from the following representation theoretic interpretation of the liminal squarefree type polynomials $T_{\lambda,\infty}^\sfr(q)$. Recall that for a partition $\lambda$ the liminal squarefree type polynomial $T_{\lambda,\infty}^\sfr(q)$ is defined by
\[
    T_{\lambda,\infty}^\sfr(q) = \lim_{n\rightarrow\infty} T_{\lambda,n}^\sfr(q),
\]
where $T_{\lambda,n}^\sfr(q)$ is the number of monic squarefree polynomials in $\FF_q[x_1,x_2,\ldots, x_n]$ with factorization type $\lambda$. 

\begin{thm}
\label{thm rep interp}
Let $\lambda \vdash d$ be a partition, and let $\sigma_d^k$ be the character of the $S_d$-representation $\Sigma_d^k$ defined in \eqref{eqn rep def}. Then
\[
    T_{\lambda,\infty}^\sfr(q) = \frac{1}{z_\lambda(1 - q)^d}\sum_{k=0}^{d-1} (-1)^k \sigma_d^k(\lambda) q^{d-k},
\]
where $z_\lambda = \prod_{j\geq 1} j^{m_j(\lambda)}m_j(\lambda)!$ is the number of permutations in $S_d$ commuting with a permutation of cycle type $\lambda$.
\end{thm}

\begin{proof}
Let $\psi_d^k$ be the character of the $S_d$-representation $H^{2k}(\PConf_d(\RR^3),\QQ)$. In \cite[Thm. 2.1]{Hyde} we showed that for all partitions $\lambda\vdash d$,
\[
    T_{\lambda,1}(q) = \frac{1}{z_\lambda}\sum_{k=0}^{d-1} \psi_d^k(\lambda)q^{d-k}.
\]
Thus, Theorem \ref{thm main reciprocity} gives
\begin{align*}
    T_{\lambda,\infty}^\sfr(q) &= (-1)^{\ell(\lambda)} T_{\lambda,1}\Big(\tfrac{1}{1 - \frac{1}{q}}\Big)\\
    &= \frac{1}{z_\lambda}\sum_{j=0}^{d-1} (-1)^{\ell(\lambda)}\psi_d^j(\lambda)\bigg(\frac{1}{1 - \frac{1}{q}}\bigg)^{d-j}\\
    &= \frac{1}{z_\lambda(1 - q)^d}\sum_{j=0}^{d-1} (-1)^{d - \ell(\lambda)} \psi_d^j(\lambda) q^{d-j} (q - 1)^j\\
    &= \frac{1}{z_\lambda(1 - q)^d}\sum_{j=0}^{d-1} \sgn_d(\lambda) \psi_d^j(\lambda) q^{d-j} \sum_{k=0}^j (-1)^k\binom{j}{k} q^{j-k}\\
    &= \frac{1}{z_\lambda(1 - q)^d}\sum_{k=0}^{d-1} (-1)^k \bigg(\sum_{j=k}^d \binom{j}{k}\sgn_d(\lambda) \psi_d^j(\lambda)\bigg) q^{d-k}\\
    &= \frac{1}{z_\lambda(1 - q)^d}\sum_{k=0}^{d-1}(-1)^k \sigma_d^k(\lambda) q^{d-k}.
\end{align*}
\end{proof}

\noindent We now prove Theorem \ref{thm rep interp}.

\begin{proof}
Since $P$ depends only on factorization type, the limiting first moment of $P$ may be rewritten as
\[
    \lim_{n\rightarrow\infty} \sum_{f\in \poly_{d,n}^\sfr(\FF_q)} P(f) = \lim_{n\rightarrow\infty} \sum_{\lambda \vdash d} P(\lambda) T_{\lambda,n}^\sfr(q) = \sum_{\lambda\vdash d} P(\lambda)T_{\lambda,\infty}^\sfr(q).
\]
Then Theorem \ref{thm rep interp} implies
\begin{align*}
    \sum_{\lambda\vdash d} P(\lambda)T_{\lambda,\infty}^\sfr(q) &= \sum_{\lambda\vdash d} \frac{1}{z_\lambda(1 - q)^d}\sum_{k=0}^{d-1} (-1)^k P(\lambda)\sigma_d^k(\lambda) q^{d-k}\\
    &=  \frac{1}{(1 - q)^d}\sum_{k=0}^{d-1} (-1)^k \sum_{\lambda\vdash d}\frac{P(\lambda)\sigma_d^k(\lambda)}{z_\lambda} q^{d-k}\\
    &= \frac{1}{(1 - q)^d}\sum_{k=0}^{d-1} (-1)^k \langle P, \sigma_d^k \rangle q^{d-k}.
\end{align*}
\end{proof}

The coefficients of $T_{\lambda,1}^\sfr(q)$ also have representation theoretic interpretations, which suggests that we might hope for a version of Theorem \ref{thm rep interp} for the limiting first moments of factorization statistics on $\poly_{d,n}(\FF_q)$. However, computations show that the coefficients of $T_{\lambda,\infty}(q)$ are determined by virtual characters, unlike those of $T_{\lambda,\infty}^\sfr(q)$. Since this is what we would expect for an arbitrary class function valued in $\frac{1}{z_\lambda}\ZZ$ we do not pursue it.

In \cite{Hyde} we pose the question of finding a geometric interpretation of Theorem \ref{thm univariate twisted gl} which explains the connection between the configuration space $\PConf_d(\RR^3)$ and factorization statistics of degree $d$ polynomials over $\FF_q$. Going further, we would like to know any conceptual explanation for Theorem \ref{thm rep interp}, be it geometric or combinatorial. The sequence of representations $\Sigma_d^k$ is unfamiliar to us; some basic properties are collected below in Proposition \ref{liminal properties} with the hope that they may be recognized by the reader.

The representation theoretic interpretation of the coefficients of $T_{\lambda,\infty}^\sfr(q)$ was discovered empirically by the author pursuing generalizations of squarefree splitting measures to multivariate polynomials. It was in the course of trying to establish this connection with representation theory that the liminal reciprocity and all the results of \cite{Hyde} were found.

\subsection{Example}

We demonstrate the liminal reciprocity identity of Theorem \ref{thm main reciprocity} by computing the expected value of the sign statistic $\sgn_d$ on degree $d$ univariate polynomials $\poly_{d,1}(\FF_q)$ and the limiting expected value of $\sgn_d$ on squarefree degree $d$ polynomials $\poly_{d,\infty}^\sfr(\FF_q)$.

Let $\sgn_d$ be the sign character of $S_d$. Note that $\sgn_d(\lambda) = (-1)^d(-1)^{\ell(\lambda)}$, where $\ell(\lambda) = \sum_{j\geq 1} m_j(\lambda)$ is the number of parts of $\lambda$. Recall that $P_{d,n}(q) = |\poly_{d,n}(\FF_q)|$ and $P_{d,n}^\sfr(q) = |\poly_{d,n}^\sfr(\FF_q)|$.

\begin{prop}
\label{sgn prop}
Let $d\geq 1$.
\begin{enumerate}
    \item The expected value $E_{d,1}(\sgn_d)$ of the sign statistic on the set $\poly_{d,1}(\FF_q)$ is given by
    \[
        E_{d,1}(\sgn_d) := \frac{1}{P_{d,1}(q)}\sum_{f\in \poly_{d,1}(\FF_q)} \sgn_d(f) =  \frac{1}{q^{\lfloor d/2 \rfloor}}.
    \]
    \item The limiting expected value $E_{d,\infty}^\sfr(\sgn_d)$ of the sign statistic on the set $\poly_{d,n}^\sfr(\FF_q)$ as $n\rightarrow\infty$ is given by
    \[
        E_{d,\infty}^\sfr(\sgn_d) := \lim_{n\rightarrow\infty} \frac{1}{P_{d,n}^\sfr(q)}\sum_{f\in \poly_{d,n}^\sfr(\FF_q)} \sgn_d(f) = \left(\frac{1}{1-\frac{1}{q}}\right)^{\lfloor d/2 \rfloor},
    \]
    where the limit is taken $1/q$-adically.
\end{enumerate}
\end{prop}

\begin{proof}
(1) Since $\sgn_d(f)$ depends only on the factorization type of $f$ we have
\[
    \sum_{f\in \poly_{d,1}(\FF_q)} \sgn_d(f) = \sum_{\lambda \vdash d} \sgn(\lambda) T_{\lambda,1}(q).
\]
Theorem \ref{thm main reciprocity} gives the identity
\[
    (-1)^{\ell(\lambda)}T_{\lambda,1}(q) = T_{\lambda,\infty}^\sfr\Big(\tfrac{1}{1-\frac{1}{q}}\Big),    
\]
from which we deduce for each $d\geq 1$
\begin{align*}
    \sum_{\lambda \vdash d} \sgn(\lambda) T_{\lambda,1}(q) &= \sum_{\lambda \vdash d} (-1)^d (-1)^{\ell(\lambda)} T_{\lambda,1}(q)\\
    &= \sum_{\lambda \vdash d} (-1)^d T_{\lambda,\infty}^\sfr\Big(\tfrac{1}{1-\frac{1}{q}}\Big)\\
    &=(-1)^d P_{d,\infty}^\sfr\Big(\tfrac{1}{1-\frac{1}{q}}\Big).
\end{align*}
Theorem \ref{thm converge} (2) tells us
\[
    P_{d,\infty}^\sfr(q) = (-1)^d \Big(\tfrac{1}{1-\frac{1}{q}}\Big)^{\lfloor \frac{d +1}{2} \rfloor}.
\]
Thus,
\[ 
    \sum_{\lambda \vdash d} \sgn_d(\lambda) T_{\lambda,1}(q) = (-1)^d P_{d,\infty}^\sfr\Big(\tfrac{1}{1-\frac{1}{q}}\Big) = q^{\lfloor \frac{d + 1}{2}\rfloor}.
\]
Since $P_{d,1}(q) = q^d$ and $d - \lfloor (d+1)/2 \rfloor = \lfloor d/2 \rfloor$ it follows that
\[
    E_{d,1}(\sgn_d) = \frac{1}{P_{d,1}(q)}\sum_{f\in \poly_{d,1}(\FF_q)} \sgn(f) = \frac{1}{q^{\lfloor d/2 \rfloor}}.
\]

\noindent (2) For each $n \geq 1$,
\[
    E_{d,n}^\sfr(\sgn_d) := \frac{1}{P_{d,n}^\sfr(q)}\sum_{f\in \poly_{d,n}^\sfr(\FF_q)} \sgn_d(f) =  \frac{1}{P_{d,n}^\sfr(q)}\sum_{\lambda \vdash d} \sgn(\lambda) T_{\lambda,n}^\sfr(q).
\]
Taking a limit as $n \rightarrow\infty$,
\[
    E_{d,\infty}^\sfr(\sgn_d) = \frac{1}{P_{d,\infty}^\sfr(q)} \sum_{\lambda\vdash d} \sgn_d(\lambda) T_{\lambda,\infty}^\sfr(q).
\]
Theorem \ref{thm main reciprocity} gives us
\[
    (-1)^{\ell(\lambda)}T_{\lambda,\infty}^\sfr(q) = T_{\lambda,1}\Big(\tfrac{1}{1-\frac{1}{q}}\Big).
\]
Therefore,
\begin{align*}
    \sum_{\lambda \vdash d}\sgn_d(\lambda)T_{\lambda,\infty}^\sfr(q) &= \sum_{\lambda \vdash d} (-1)^d(-1)^{\ell(\lambda)}T_{\lambda,\infty}^\sfr(q)\\
    &=\sum_{\lambda\vdash d} (-1)^d T_{\lambda,1}\Big(\tfrac{1}{1-\frac{1}{q}}\Big)\\ 
    &= (-1)^d\Big(\tfrac{1}{1-\frac{1}{q}}\Big)^d.
\end{align*}
Since $P_{d,\infty}^\sfr(q) = (-1)^d\Big(\tfrac{1}{1-\frac{1}{q}}\Big)^{\lfloor(d+1)/2\rfloor}$ and $d - \lfloor (d + 1)/2 \rfloor = \lfloor d/2 \rfloor$ we conclude that
\[ 
    E_{d,\infty}^\sfr(\sgn_d) = \frac{1}{P_{d,\infty}^\sfr(q)}\sum_{\lambda \vdash d}\sgn_d(\lambda)T_{\lambda,\infty}^\sfr(q) = \bigg(\frac{1}{1-\frac{1}{q}}\bigg)^{\lfloor d/2 \rfloor}.
\] 
\end{proof}

Note that Theorem \ref{thm univariate twisted gl} (1) tells us that
\[
    E_{d,1}(\sgn_d) = \sum_{k=0}^{d-1}\frac{\langle \sgn_d,\psi_d^k\rangle}{q^k}.
\]
Comparing this with Proposition \ref{sgn prop} (1) it follows that $H^{2k}(\PConf_d(\RR^3),\QQ)$ has a one dimensional $\sgn_d$ component when $k = \lfloor d/2 \rfloor$ and no $\sgn_d$ component for any other value of $k$.

The sign function $\sgn_d$ is closely related to the \emph{Liouville function} $\lambda$ studied by Carlitz \cite{CarlitzA, CarlitzB} in the context of polynomials in $\FF_q[x]$. In particular, if $f(x) \in \poly_{d,1}(\FF_q)$
\[
    \lambda(f) = (-1)^d\sgn_d(f).
\]
Carlitz \cite[(ii) pg. 121]{CarlitzA}\cite[Sec. 3]{CarlitzB} computes the first moment of the Liouville function using zeta functions. Proposition \ref{sgn prop} may also be deduced from his result. We thank Ofir Gorodetsky for bringing this work to our attention.

\subsection{The $S_d$-representations $\Sigma_d^k$}
\label{section liminal reps}

Theorem \ref{thm liminal twisted GL} relates the limiting first moments of factorization statistics on squarefree polynomials with a family of symmetric group representations $\Sigma_d^k$. Recall that
\[
    \Sigma_d^k = \bigoplus_{j=k}^{d-1} \sgn_d \otimes H^{2j}(\PConf_d(\RR^3),\QQ)^{\oplus \binom{j}{k}}.
\]
We conclude with Proposition \ref{liminal properties} which records some observations about the representations $\Sigma_d^k$.

\begin{prop}
\label{liminal properties}
Let $\sigma_d^k$ be the character of $\Sigma_d^k$. Then
\begin{enumerate}
    \item The dimension of $\Sigma_d^k$ is
    \[ 
        \dim \Sigma_d^k = \sum_{i=k}^{d-1}\stir{d}{d-i}\binom{i}{i - k},
    \] 
    where $\stir{m}{n}$ is an unsigned Stirling number of the first kind.
    
    \item The representation
    \[
        \bigoplus_{k=0}^{d-1} \Sigma_d^k 
    \]
    has dimension $(2d - 1)!! = (2d - 1)(2d - 3)\cdots 3 \cdot 1$.
    
    \item $\Sigma_d^0$ is isomorphic to the regular representation $\QQ[S_d]$.
\end{enumerate}
\end{prop}

Note that the sequence $\dim \Sigma_d^k$ appears as A088996 in the \emph{Online Encyclopedia of Integer Sequences} \cite{oeis}.

\begin{proof}
(1) The dimension of a representation is given by evaluating its character on the identity, hence
\[
    \dim \Sigma_d^k = \sigma_d^k((1^d)).
\]
Theorem \ref{thm rep interp} implies that
\[
    T_{(1^d),\infty}^\sfr(q) = \frac{1}{d!(1 - q)^d}\sum_{k=0}^{d-1}(-1)^k\sigma_d^k((1^d))q^{d-k}.
\]
On the other hand, we may compute $T_{(1^d),\infty}^\sfr(q)$ directly as
\[
    T_{(1^d),\infty}^\sfr(q) = \binom{M_{d,\infty}(q)}{d} = \binom{-\frac{1}{1 - \frac{1}{q}}}{d}.
\]
The \emph{unsigned Stirling numbers of the first kind} are defined to be the coefficients in the expansion of a binomial coeffcient $\binom{x}{d}$,
\[
    \binom{x}{d} = \frac{1}{d!}\sum_{k=0}^{d-1} (-1)^k \stir{d}{d-k}x^{d-k}.
\]
Thus,
\begin{align*}
    T_{(1^d),\infty}^\sfr(q) &= \frac{1}{d!}\sum_{i=0}^{d-1} (-1)^i \stir{d}{d-i}\left(-\frac{1}{1 - \frac{1}{q}}\right)^{d-i}\\
    &= \frac{1}{d!(1 - q)^d}\sum_{i=0}^{d-1} (-1)^i \stir{d}{d-i}q^{d-i}(1 - q)^i\\
    &= \frac{1}{d!(1 - q)^d}\sum_{i=0}^{d-1}\sum_{j=0}^i (-1)^{i+j} \stir{d}{d-i}\binom{i}{j}q^{d-(i - j)}.
\end{align*}
Let $k = i - j$ and write the sum in terms of $i$ and $k$ to get
\[
    T_{(1^d),\infty}^\sfr(q) = \frac{1}{d!(1 - q)^d}\sum_{k=0}^{d-1} (-1)^k \left(\sum_{i=k}^{d-1}\stir{d}{d-i}\binom{i}{i - k}\right)q^{d-k}.
\]
Comparing coefficients in our two expressions for $T_{(1^d),\infty}^\sfr(q)$ we conclude that
\[
    \dim \Sigma_d^k = \sigma_d^k((1^d)) = \sum_{i=k}^{d-1}\stir{d}{d-i}\binom{i}{i - k}.
\]

(2) Let $\psi_d^k$ be the character of $H^{2k}(\PConf_d(\RR^3),\QQ)$. Then using the definition of $\Sigma_d^k$ and switching the order of summation we have
\[
    \sum_{k=0}^{d-1} \sigma_d^k((1^d)) = \sum_{k=0}^{d-1} \sum_{j=k}^d \binom{j}{k} \psi_d^j((1^d))
    = \sum_{j=0}^{d-1} \sum_{k=0}^j \binom{j}{k} \psi_d^j((1^d))
    = \sum_{j=0}^{d-1} 2^j \psi_d^j((1^d)).
\]
Note that by Theorem \ref{thm univariate twisted gl} (1),
\begin{equation}
\label{eqn dim}
   \sum_{j=0}^{d-1} \frac{\psi_d^j((1^d))}{q^j} = d!\frac{T_{(1^d),1}(q)}{q^d} = \frac{d!}{q^d}\binom{q + d - 1}{d}.
\end{equation} 
Evaluating \eqref{eqn dim} at $q = \frac{1}{2}$ implies
\[
    \sum_{j=0}^{d-1} 2^j\psi_d^j((1^d)) = 2^d d! \binom{d - \frac{1}{2}}{d} = (2d - 1)(2d - 3)\cdots 3 \cdot 1 = (2d-1)!!.
\]
Therefore $\dim \bigoplus_{k=0}^d \Sigma_d^k = (2d-1)!!$.\\

(3) By definition we have
\[
    \Sigma_d^0 \cong \sgn_d \otimes \bigoplus_{j=0}^{d-1} H^{2j}(\PConf_d(\RR^3),\QQ).
\]
In \cite[Thm. 2.8]{Hyde} we showed that
\[
    \bigoplus_{j=0}^{d-1} H^{2j}(\PConf_d(\RR^3),\QQ) \cong \QQ[S_d],
\]
where $\QQ[S_d]$ is the regular representation. The claim follows from 
\[
    \sgn_d \otimes \QQ[S_d] \cong \QQ[S_d].
\]
\end{proof}

\end{document}